\documentclass{siamltex}
\usepackage{amsmath,amsfonts,amssymb}
\usepackage{latexsym}
\usepackage{epsfig,overpic}
\usepackage{stmaryrd}
\usepackage{color}
\usepackage{graphicx}

\newcommand{\R}{\mathbb R}

\newcommand{\bH}{\mathbf H}
\newcommand{\bI}{\mathbf I}

\newcommand{\bP}{\mathbf P}

\newcommand{\bn}{\mathbf n}
\newcommand{\bp}{\mathbf p}

\newcommand{\bs}{\mathbf s}

\newcommand{\bx}{\mathbf x}

\newcommand{\T}{\mathcal T}

\newcommand{\uDelta}{{\Delta}_{\Gamma}}
\newcommand{\unabla}{{\nabla}_{\Gamma}}

\newcommand{\rd}{\mathrm{d}}

\definecolor{darkred}{rgb}{.7,0,0}

\definecolor{green}{rgb}{0,0.7,0}

\def\dO{{\partial\Omega} }
\def\div{\operatorname{div} }

\newtheorem{remark}{Remark}

\begin{document}

\title{Non-degenerate Eulerian finite element method for solving PDEs on surfaces}
\author{Alexey Y. Chernyshenko\thanks{
Institute of Numerical Mathematics, Russian Academy of Sciences, Moscow 119333}
\and
Maxim A. Olshanskii\thanks{Department of Mathematics, University of Houston, Houston, Texas 77204 and Department of Mechanics and Mathematics, Moscow State University, Moscow 119899 {\tt molshan@math.uh.edu}}
}

\maketitle

\markboth{}{}

\begin{abstract}  The paper studies a method for solving elliptic partial differential equations
posed on hypersurfaces in $\mathbb{R}^N$, $N=2,3$. The method builds upon the formulation introduced in Bertalmio et al., J. Comput. Phys., 174 (2001), 759--780., where  a surface equation is extended to a neighborhood of the surface.
The resulting degenerate PDE is then solved in one dimension higher, but can be solved on a mesh
that is unaligned to the surface.
We introduce another extended formulation, which leads to  uniformly elliptic (non-degenerate) equations
in a bulk domain containing the surface. We apply a finite element method to solve this extended PDE
and prove the convergence of finite element solutions restricted to the surface to the solution of the
original surface problem. Several numerical examples illustrate the properties of the method.
 \end{abstract}




\section{Introduction}
Partial differential equations posed on surfaces arise in mathematical models for many natural phenomena:
diffusion along grain boundaries \cite{grain2}, lipid interactions in biomembranes \cite{membrains1}, and transport of surfactants on multiphase flow interfaces \cite{GrossReuskenBook}, as well as in many engineering and bioscience applications: vector field visualization \cite{vector}, textures synthesis \cite{texture1}, brain warping \cite{imaging2}, fluids in lungs \cite{lungs} among others.
Thus, recently there has been a significant increase of interest  in developing and analyzing numerical methods for  PDEs on surfaces.

The development of numerical methods based on surface triangulation started with the paper  of Dziuk  \cite{Dziuk88}. In this class of methods, the surface is approximated by a family of consistent  regular triangulations. It is typically assumed that all vertices in the triangulations lie on the surface.
In  \cite{Dziuk07} the  method from \cite{Dziuk88} was combined with Lagrangian surface tracking and  was generalized to equations on  evolving surfaces. To avoid surface triangulation and remeshing, another approach was taken in \cite{BCOS01}: It was proposed to extend a partial differential equation from the surface to a set of positive Lebesgue measure in $\R^3$. The resulting PDE is then solved in one dimension higher, but can be solved on a mesh that is unaligned to the surface. If the surface evolves, the approach allows to avoid a Lagrangian description of the surface evolution and is commonly referred to as Eulerian approach~\cite{XuZhao}.

Despite clear advantages, the method from~\cite{BCOS01} has a number of drawbacks, see \cite{DDEH,Greer06} for the careful account of pros and cons of the approach. In particular, the resulting bulk elliptic or parabolic PDEs are degenerate, since no diffusion acts in the direction normal to the surface. Setting  boundary conditions in a numerical method for such problem is another issue.
An attempt to overcome the degeneracy and related issues was made in \cite{Greer06}, where a modification
of the method was introduced for parabolic problems.

The method for surface equations in the present paper benefits from the modification introduced by Greer in \cite{Greer06} of the formulation from \cite{BCOS01}. We develop a new extended formulation, which leads to a uniformly elliptic equations
in a bulk domain containing the surface. The formulation preserves all advantages of the one from
 \cite{Greer06}, but adds diffusion in the normal direction in a more consistent way and avoids introducing
 additional parameters. Further, we consider a Galerkin (finite element) method for solving the extended equation. Taking the advantage of the non-degeneracy of the  extended formulation, we prove error estimates  in the $L^2$ and $L^\infty$ surface norms. To the best of our knowledge, such estimates were previously unknown  for an Eulerian surface finite element method based on extension of PDE from the surface.

Another Eulerian finite element method for elliptic equations posed on surfaces was introduced in~\cite{ORG09,OlsR2009}. That method does not use an extension of the surface partial differential equation.  It is instead based on a restriction (trace) of the outer  finite element spaces to a surface.
It is not the intension of this paper to compare this different approaches.

The remainder of the paper is organized as follows.  Section~\ref{s_prel} collects some necessary definitions and preliminary results. In section~\ref{s_form}, we recall the extended PDE approach from \cite{BCOS01} and
the modified formulation from~\cite{Greer06}. Further, we introduce the new formulation and show its well-posedness. In section~\ref{s_FEM}, we consider a finite element method and prove error estimates. Section~\ref{s_numer} presents the result of several numerical experiments that demonstrate the performance of the finite element method. Finally, section~\ref{s_concl} collects some closing remarks.


\section{Preliminaries}\label{s_prel}
We assume that $\Omega$ is an open subset in $\R^N$, $N=2,3$ and $\Gamma$ is a connected $C^2$ compact
hypersurface contained in $\Omega$.
For a sufficiently smooth function $g: \Omega \to \R$ the tangential
gradient (along $\Gamma$) is defined by
\begin{equation} \label{tangderi}
  \unabla g=\nabla g - \nabla g \cdot \bn_\Gamma \, \bn_\Gamma.
\end{equation}
By $\uDelta$ we denote  the Laplace--Beltrami operator on
$\Gamma$, $\uDelta=\nabla_\Gamma\cdot\nabla_\Gamma$.

This paper deals with elliptic  equations posed on $\Gamma$. As a basic elliptic equation, we
consider the  Laplace--Beltrami  problem:
\begin{equation}\label{LBeq}
-\uDelta u + \alpha\, u =f\quad\text{on}~\Gamma,
\end{equation}
with some $\alpha\in L^\infty(\Gamma)$.
The corresponding weak form of \eqref{LBeq} reads: For given $f\in L^2(\Gamma)$  determine $u \in H^1(\Gamma)$ such that
\begin{equation}\label{weak}
\int_\Gamma \unabla u \unabla v+\alpha\,uv\,\rd\bs= \int_\Gamma f v\, \rd\bs\qquad\text{for all}~~v\in H^1(\Gamma).
\end{equation}
For the well-possedness of \eqref{weak}, it is sufficient to assume $\alpha$  to be strictly positive on a subset of $\Gamma$ with
positive surface measure:
\begin{equation}\label{cA}
\mathcal{A}:=\operatorname{meas}_{\bs}\{\bx\in\Gamma\,:\, \alpha(\bx)\ge {\alpha}_0\}>0,
\end{equation}
with some $\alpha_0>0$.
In this case,  the following Friedrich's type inequality~\cite{Sobolev} (see, also Lemma~3.1 in \cite{ORG09}):
\begin{equation}\label{Fried}
\|v\|_{L^2(\Gamma)}^2 \le C_F(\|\nabla_\Gamma v\|_{L^2(\Gamma)}^2+\|\sqrt{\alpha} v\|_{L^2(\Gamma)}^2)\quad
\forall~v\in H^1(\Gamma)
\end{equation}
holds with a constant $C_F$ dependent of $\alpha_0$ and $\mathcal{A}$.

The solution $u$ to \eqref{weak} is unique and  satisfies $u\in H^2(\Gamma)$, with $\|u\|_{H^2(\Gamma)}\le
c\|f\|_{L^2(\Gamma)}$ and a constant $c$ independent of $f$, cf.~\cite{Dziuk88}.
We remark that the case $\alpha=0$ is also covered by the analysis of the paper. In this case, the Friedrich's  inequality \eqref{Fried} holds for all $v\in H^1(\Gamma)$ with zero mean.
Hence, if $\alpha=0$, we assume $\int_\Gamma f\,\rd\bs=0$ and look for the unique solution to \eqref{weak} satisfying  $\int_\Gamma u\,\rd\bs=0$.
\medskip

Denote by $\Omega_d$ a domain consisting of all points within a distance from $\Gamma$ less than some $d>0$:
\begin{equation}\label{Omega_d}
\Omega_d = \{\, \bx \in  \R^N~:~{\rm dist}(\bx,\Gamma) < d\, \}.
\end{equation}
Let $\phi: \Omega_d \rightarrow \R$ be the
signed distance function, $|\phi(x)|:={\rm dist}(\bx,\Gamma)$ for all
$\bx \in \Omega_d$. The surface $\Gamma$ is the zero level
set of $\phi$:
\begin{equation}\label{lsf}
\Gamma=\{\bx\in\mathbb{R}^N\,:\,\phi(\bx)=0\}.
\end{equation}
We may assume $\phi < 0$ on the interior of $\Gamma$  and $\phi >0$ on the exterior.
We define $\bn(\bx):=\nabla \phi(\bx)$ for all
$\bx \in \Omega_d$. Thus,
$\bn$ is the outward normal vector on $\Gamma$, $\bn_{\Gamma}
=\nabla \phi$ on $\Gamma$,  and $|\bn (\bx)|=1$ for all $\bx\in \Omega_d$.  The Hessian of $\phi$ is denoted by $\bH$:
\begin{equation} \label{Hessian}
  \bH(\bx)=D^2\phi(\bx) \in \R^{3 \times 3} \quad \text{for all} ~~\bx \in \Omega_d.
\end{equation}
The eigenvalues of $\bH(\bx)$ are denoted by $\kappa_1(\bx),
\kappa_2(\bx)$, and 0. For  $\bx \in \Gamma$, the eigenvalues $\kappa_i(\bx)$, $i=1,2$, are
the principal curvatures.

We will need the orthogonal projection
\[
 \bP(\bx)= \bI-\bn(\bx)\otimes\bn(\bx) \quad \text{for all}~~\bx \in \Omega_d.
\]
Note that the tangential gradient can be written as $\unabla g(\bx)= \bP \nabla g(\bx)$ for
$\bx \in \Gamma$.
We introduce a locally
orthogonal coordinate system by using the  projection $\bp:\, \Omega_d \rightarrow
\Gamma$:
\[
 \bp(\bx)=\bx-\phi(\bx)\bn(\bx) \quad \text{for all}~~\bx \in \Omega_d.
\]
Assume that the decomposition $\bx=\bp(\bx)+ \phi(\bx)\bn(\bx)$ is unique for all $\bx \in \Omega_d$.
We shall use an extension operator  defined as follows. For a  function $v$ on $\Gamma$
we define
\begin{equation} \label{extension}
 v^e(\bx):= v(\bp(\bx)) \quad \text{for all}~~\bx \in \Omega_d.
\end{equation}
Thus, $ v^e$ is the extension of $v$ along normals on $\Gamma$, it satisfies $\bn\cdot \nabla v^e=0$ in $\Omega_d$, i.e., $v^e$ is constant along normals to $\Gamma$.

\section{Extended surface PDEs} \label{s_form}
In this section, we define an extension of the surface PDE \eqref{LBeq} to a neighborhood of $\Gamma$.
Recalling \eqref{lsf} and   $\unabla u=\bP\nabla u$ on $\Gamma$, we write the weak formulation of  \eqref{LBeq} on the zero level set of $\phi$:
\[
\int_{\Gamma}\unabla u\cdot\unabla v + \alpha\,uv-fv\,d\bs=\int_{\{\phi=0\}}\bP\nabla u\cdot \bP\nabla v + c\,uv-fv\,d\bs=0.
\]
The idea of \cite{BCOS01} is to extend \eqref{weak}, with the help of globally defined quantities $\bn$
  and $\bP$, to every level set of $\phi$ intersecting $\Omega_d$: Find $u\in H_P$ such that
 \begin{equation}\label{weakV}
\begin{split}
0&=\int_{-d}^{+d}\int_{\{\phi=r\}}\bP\nabla u\cdot \bP\nabla v + \alpha^e\,uv-f^ev\,d\bs\, dr\\
&=\int_{\Omega_d}(\bP\nabla u\cdot \bP\nabla v + \alpha^e\,uv-f^ev)|\nabla\phi| d\bx \quad \hbox{ for all } v\in H_P,
\end{split}
\end{equation}
where
\[
H_P=\{v\in L^2(\Omega_d)\,:\,\bP\nabla v\in (L^2(\Omega_d))^{N}\}.
\]
The above weak formulation was shown to be well-posed in \cite{Burger}. The surface  equation \eqref{weak} is embedded in \eqref{weakV} and the solution on every level set of $\phi$ does not depend
on a data in a neighborhood of this level set (indeed, one can consider \eqref{weakV} as a collection of
of mutually independent surface problems on every level set of $\phi$). Hence, restricted to $\Gamma$, smooth solution to   \eqref{weakV} solves the original Laplace-Beltrami  problem \eqref{LBeq}. With no ambiguity,
we shall denote by $u$ both the solutions to surface and extended problems.

The corresponding strong formulation of  \eqref{weakV} is
 \begin{equation}
\label{1.2}
-|\nabla\phi|^{-1}\div|\nabla\phi| \bP\nabla u+ \alpha^e\,u=f^e\quad  \hbox{ in } \Omega_d.
\end{equation}
We note that \eqref{weakV}  and \eqref{1.2} are the valid extensions of \eqref{weak} and \eqref{LBeq}
if $\phi$ is an arbitrary smooth level set function with $\nabla\phi\neq0$, not necessarily a signed distance function,
and $\alpha^e, f^e$ are not necessarily constant along normal directions. If the boundary of the volume domain $\Omega_d$ is not a level set of $\phi$, then \eqref{1.2} should be complemented with boundary conditions. This can be natural boundary conditions
\begin{equation}
\label{1.2bc}
(\bP\nabla u)\cdot\bn_\dO=0\quad  \hbox{ on } \dO_d,
\end{equation}
where $\bn_\dO$ is the outward normal vector to $\dO_d$.

The major numerical advantage of the extended formulation is that one may apply standard
discretization methods to solve  \eqref{1.2}--\eqref{1.2bc} in the volume domain $\Omega_d$ (e.g., a finite difference method on Cartesian grids) and further take the trace of  computed solutions on $\Gamma$ (or on a approximation of $\Gamma$). Numerical experiments from \cite{BCOS01,Burger,Greer06,XuZhao} suggest that these traces of numerical solutions  are reasonably good approximations to the solution of the surface problem  \eqref{LBeq}. The analysis of the method is still limited: Error estimates for finite element methods for \eqref{weakV} are shown in~\cite{Burger,DDEH}.
Error estimate in \cite{Burger} is established only in the integral volume norm
\[
\|v\|_{H_P}^2:=\|v\|^2_{L^2(\Omega_d)}+\|\bP\nabla v\|^2_{L^2(\Omega_d)},
\]
rather than in a surface norm for $\Gamma$. In \cite{DDEH} the first order convergence was proved in the surface $H^1$ norm, if the band width $d$ in \eqref{Omega_d} is of the order of mesh size and if a quasi-uniform triangulation of $\Omega$ is assumed. For linear elements this estimate is of the optimal order.

Although numerically convenient, the extended formulation has a number of disadvantages, as noted already
in   \cite{BCOS01} and reviewed in \cite{DDEH,Greer06}. The volume formulation \eqref{1.2} is defined in a domain in one dimension higher than the surface equation. This leads to involving extra degrees of freedom in numerical method. If $\Omega_d$ is a narrow band around $\Gamma$, then handling boundary conditions \eqref{1.2bc} may effect the quality of the discrete solution. This can be an issue for grids not aligned with a level set
of $\phi$ on $\dO_d$.  Numerical stability calls for the extension of data satisfying \eqref{extension}; and in time-stepping schemes for parabolic problems, one needs the intermittent re-initialization of $u$ by
re-extending it from $\Gamma$ according to \eqref{extension}.
Another issue of the extended formulation \eqref{1.2} is that the second order term is degenerate, since no diffusion acts in the direction normal to level sets of $\phi$.  Numerical solution of degenerate elliptic and parabolic equations is not a very  well understood subject.

An effort to overcome the degeneracy and some related issues of the approach from \cite{BCOS01} was done by Greer in~\cite{Greer06}, where the heat equation
 \begin{equation}
\label{heat}
\frac{\partial u}{\partial t} - \Delta_\Gamma u=0,\quad u|_{t=0}=u_0
\end{equation}
on a stationary surface $\Gamma$ was studied. In the method from  \cite{Greer06},
 one extends \eqref{heat} to a   neighborhood of $\Gamma$ ensuring the following properties hold:
\begin{enumerate}
\item $\phi$ is the singed distance function;
\item $u_0$ is extended to the neighborhood of $\Gamma$  according to \eqref{extension}, i.e., constant alone normals;
\item The projection $\bP$ is changed to the (non-orthogonal) scaled projection
\begin{equation}\label{tildeP}
\widetilde{\bP}:=(\bI-\phi \bH)^{-1} \bP
\end{equation}
on  tangential planes of the level sets of $\phi$. The bulk domain  $\Omega_d$ is assumed such that the modified projection $\widetilde{\bP}$  is well defined. For a smooth $\Gamma$, this can be always enured by choosing small enough $d>0$.
\end{enumerate}
With the above assumptions, the solution to the extended heat equation
 \begin{equation}
\label{heatExt}
\frac{\partial u}{\partial t} - (\widetilde{\bP}\nabla)\cdot\widetilde{\bP}\nabla u=0,\quad u|_{t=0}=u_0^e\quad \text{in}~~\Omega_d
\end{equation}
is proved to be constant in normal directions:
\begin{equation}\label{const}
(\bn\cdot\nabla) u=0\quad\text{in}~\Omega_d
\end{equation}
for all $t>0$.

The property \eqref{const} is crucial, since it allows to add diffusion in the normal direction without altering solution. Doing this, one obtains a non-degenerated elliptic operator. Thus, instead of \eqref{heatExt}  it was suggested in~\cite{Greer06} to consider the parabolic problem
 \begin{equation}
\label{heatExt1}
\frac{\partial u}{\partial t} - (\widetilde{\bP}\nabla)\cdot\widetilde{\bP}\nabla u-c^2_n\div(\bn\otimes\bn)\nabla u=0,\quad u|_{t=0}=u_0\quad \text{in}~~\Omega_d,
\end{equation}
 with a coefficient  $c^2_n$. For the planar case, $\Omega_d\in\mathbb{R}^2$, it was recommended to set $c_n=(1-\phi\kappa_0)$,  $\kappa_0=\kappa(p(\bx))$, $\kappa$ is the curvature of $\Gamma$ ($\Gamma$ is a curve in this case).  For the case of surfaces embedded in $\mathbb{R}^3$,
 there was no clear recommendation on $c_n$.

The above approach formally solves the problem of the degeneracy and suggests that equation \eqref{const}
on $\dO_d$ is appropriate and numerically sound boundary condition.  However,  one has to define
parameter $c_n$. Moreover, the new extended formulation involves the Hessian $\bH$. If $\Gamma$ is
given only by an approximation, for example, as the zero set of a discrete level set function $\phi_h$,
then computing (an approximation to) $\bH$ is a delicate issue, sensitive to numerical implementation.
\medskip

Below we introduce a formulation of the extended surface problem, which `automatically' generates
diffusion in the normal direction, leading to a uniformly elliptic or parabolic problem in $\Omega_d$.
The finite element method and error analysis are considered in the section~\ref{s_FEM}.  The problem of
the approximate evaluation of Hessian is addressed numerically in section~\ref{s_numer}.

\subsection{Another extension of surface PDE}
For the sake of analysis, consider the Laplace-Beltrami  equation \eqref{LBeq} rather than the surface heat equation.
We assume from now that all extensions of data from $\Gamma$ satisfy \eqref{extension}. Consider the Laplace--Beltrami equation extended from $\Gamma$ to $\Omega_d$:
 \begin{equation}
\label{1.1Ext}
-(\widetilde{\bP}\nabla)\cdot\widetilde{\bP}\nabla u+\alpha^e\, u=f^e\quad \text{in}~~\Omega_d.
\end{equation}

To  ensure that $\widetilde{\bP}$ is well-defined and equations \eqref{1.1Ext} are well-possed, it is sufficient for
the matrix $(\bI-\phi\bH)$ to be  uniformly positive definite in  $\Omega_d$. Therefore, assume $\Omega_d$ is such that
 \begin{equation}\label{ass1}
 |\phi(\bx)|=\mbox{dist}(\bx,\Gamma)\le \frac12\|\bH(\bx)\|^{-1}\quad\forall\,\bx\in\Omega_d.
\end{equation}
One can always satisfy the above restriction by choosing the band width $d$ small enough. To be more precise, from (2.5) in \cite{DD07} we have the following formula for the eigenvalues of $\bH$:
\begin{equation*}
\kappa_i(\bx)= \frac{\kappa_i(\bp(\bx))}{1 + \phi(\bx)\kappa_i(\bp(\bx))}\quad \text{for} ~\bx \in \Omega_d.
\end{equation*}
Thus, assumption \eqref{ass1} is true if  the parameter $d$ in \eqref{Omega_d}
satisfies
\[
d\le \Big(\,4\max_{\bx\in\Gamma}(|\kappa_1(\bx)|+|\kappa_2(\bx)|)\,\Big)^{-1}.
\]
Since $\Gamma\in C^2$ and $\Gamma$ is compact, the principle curvatures of $\Gamma$ are uniformly bounded and $d$ can be chosen sufficiently small positive.

The weak formulation of the problem \eqref{1.1Ext} reads: Find $u\in H_P$ satisfying
\begin{equation}\label{weakV1}
\int_{\Omega_d}\widetilde{\bP}\nabla u\cdot \widetilde{\bP}\nabla v + \alpha^e\,uv\,d\bx=\int_{\Omega_d}f^ev d\bx \quad \forall\, v\in H_P.
\end{equation}
If \eqref{ass1} holds, the existence of the unique solution to \eqref{weakV1} follows from the Lax-Milgram lemma.
If  the solution to \eqref{weakV1} is smooth, it solves the surface problem \eqref{LBeq} ($\bP=\widetilde{\bP}$ on $\Gamma$). Moreover, the  smooth solution to \eqref{weakV1} satisfies \eqref{const}. To see this, apply $(\bn\cdot\nabla)$ to equation \eqref{1.1Ext}  and use the
following commutation property (see lemma~1 in \cite{Greer06}):
\[
 (\bn\cdot\nabla)((\widetilde{\bP}\nabla)\cdot\widetilde{\bP}\nabla)=
((\widetilde{\bP}\nabla)\cdot\widetilde{\bP}\nabla)(\bn\cdot\nabla).
\]
Recalling $(\bn\cdot\nabla)\alpha^e=(\bn\cdot\nabla)f^e=0$, we get for $v_n:=(\bn\cdot\nabla)u$
\[
-(\widetilde{\bP}\nabla)\cdot\widetilde{\bP}\nabla v_n+\alpha^e\, v_n=0\quad \text{in}~~\Omega_d.
\]
The uniqueness result yields \eqref{const}.

Note that the identity $\bH \bP=\bP\bH$ implies
\begin{equation}\label{Haux}
(\bI-\phi\bH)^{-1}\bP=\bP(\bI-\phi\bH)^{-1}.
\end{equation}
Using \eqref{Haux} and $\bP^2=\bP=\bP^T$, we rewrite \eqref{weakV1} as
\begin{equation}\label{weakMIN}
\int_{\Omega_d}(\bI-\phi\bH)^{-1}\bP\nabla u\cdot (\bI-\phi\bH)^{-1}\nabla v + \alpha^e\,uv\,d\bx=\int_{\Omega_d}f^ev\,d\bx
\quad \forall\, v\in H^1(\Omega_d).
\end{equation}

Due to relations $|(\bI-\bP)\nabla u|^2=|(\bn\cdot\nabla u)\bn|^2=(\bn\cdot\nabla u)^2$, we can rewrite
equality \eqref{const} for the solution to \eqref{weakMIN} in the form
\begin{equation*}\label{Paux}
(\bI-\bP)\nabla u=0.
\end{equation*}
Thanks to \eqref{Haux}, it holds
\[
\begin{split}
(\bI-\phi\bH)^{-1}\nabla u&=\bP(\bI-\phi\bH)^{-1}\nabla u +  (\bI-\bP)(\bI-\phi\bH)^{-1}\nabla u\\
&=(\bI-\phi\bH)^{-1}\bP\nabla u + (\bI-\phi\bH)^{-1}(\bI-\bP)\nabla u\\
&=(\bI-\phi\bH)^{-1}\bP\nabla u\quad\text{for}~u~\text{solving}~\eqref{weakMIN}.
\end{split}
\]
We infer that the problem \eqref{weakMIN} can be written as follows: Find $u\in H^1(\Omega_d)$
\begin{equation}\label{weakMINa}
\int_{\Omega_d}(\bI-\phi\bH)^{-2}\nabla u\cdot \nabla v + \alpha^e\,uv\,d\bx=\int_{\Omega_d}f^ev\,d\bx
\quad \hbox{ for all } v\in H^1(\Omega_d).
\end{equation}
Now we find the strong form of \eqref{weakMINa}. To handle boundary terms arising from integration by part,
we note that $\bn=\bn_\dO$, since the boundary of the volume domain $\Omega_d$ is a level set of $\phi$. Furthermore, $\bH\bn=0$ implies $(\bI-\phi\bH)^{-1}\bn=\bn$, and so
\[
((\bI-\phi\bH)^{-1}\nabla v) \cdot \bn= (\nabla v) \cdot ((\bI-\phi\bH)^{-1}\bn)=(\bn\cdot\nabla) v.
\]
Thus, one can write \eqref{weakMINa} in the strong form:
 \begin{equation}
\label{ExtNew}
\begin{split}
-\div(\bI-\phi\bH)^{-2}\nabla u+\alpha^e\, u&=f^e\quad \text{in}~~\Omega_d\\
(\bn\cdot\nabla) u&=0\quad \text{on}~~\dO_d.
\end{split}
\end{equation}

The formulation \eqref{ExtNew} has the following advantages over \eqref{1.2},  \eqref{heatExt1}  and  \eqref{1.1Ext}:
 The equations  \eqref{ExtNew} are non-degenerate and uniformly elliptic, the extended problem has no parameters to be defined,
  the boundary conditions are given and consistent with the solution property \eqref{const}.

Regarding the well-posedness of  \eqref{ExtNew} we prove the following result.

\begin{theorem}\label{Th1} Assume \eqref{ass1}, then it holds:
\begin{description}
\item{}(i) The problem~\eqref{ExtNew} has the unique weak solution
$u\in H^1(\Omega_d)$, which satisfies $\|u\|_{H^1(\Omega_d)}\le C\,\|f^e\|_{L^2(\Omega_d)}$,
with a constant $C$ dependent only on $\alpha$, $\Gamma$ and $d$;
\item{}(ii)  Additionally assume $\Gamma\in C^3$, then $u\in H^{2}(\Omega_d)$ and
\[
\|u\|_{H^2(\Omega_d)}\le C\,\|f^e\|_{L^2(\Omega_d)},
\]
where the constant $C$ depends only on  $\alpha$, $\Gamma$ and $d$.
\end{description}
\end{theorem}
\begin{proof}
 First we  check that the bilinear form
 \[
 a(u,v):=\int_{\Omega_d}(\bI-\phi\bH)^{-2}\nabla u\cdot \nabla v + \alpha^e\,uv\,d\bx\]
 is continuous and coercive on $H^1(\Omega_d)$.

 The assumption \eqref{ass1} yields for the spectrum of the symmetric matrices:
 \[\mbox{sp}(\bI-\phi\bH)\in\left[\mbox{$\frac12,\frac32$}\right]~\Rightarrow~
 \mbox{sp}\left((\bI-\phi\bH)^{-2}\right)\in\left[\mbox{$\frac49$},4\right]\quad
 \text{for any}~\bx\in\Omega_d.\]
 Therefore, it holds
  \begin{align} \label{elp}
 \frac49\|\nabla u\|^2_{L^2(\Omega_d)}\le \int_{\Omega_d}(\bI-\phi\bH)^{-2}\nabla u\cdot \nabla u\,d\bx,
 \\ \int_{\Omega_d}(\bI-\phi\bH)^{-2}\nabla u\cdot \nabla v\,d\bx\le 4\|\nabla u\|_{L^2(\Omega_d)}\|\nabla v\|_{L^2(\Omega_d)}.
 \label{cont}
 \end{align}
 Estimates \eqref{cont} and $\|\alpha^e\|_{L^\infty(\Omega_d)}=\|\alpha\|_{L^\infty(\Gamma)}$ imply the continuity estimate
 \[
 \begin{aligned}
 a(u,v)&\le 4\|\nabla u\|_{L^2(\Omega_d)}\|\nabla v\|_{L^2(\Omega_d)}+\|\alpha^e\|_{L^\infty(\Omega_d)}\|u\|_{L^2(\Omega_d)}\|v\|_{L^2(\Omega_d)}\\
 &\le (4+\|\alpha\|_{L^\infty(\Gamma)})\|u\|_{H^1(\Omega_d)}\|v\|_{H^1(\Omega_d)}.
 \end{aligned}
 \]
Define
\[
 \mu(\bx):= \big(1-d(\bx) \kappa_1(\bx)\big)\big(1-d(\bx) \kappa_2(\bx)\big), \quad \bx \in \Omega_d .
\]
From (2.20), (2.23) in \cite{DD07} we have
$
 \mu(\bx) \rd \bx= \rd r \rd \bs(\bp(\bx))$, for $\bx \in \Omega_d$,
where $\rd \bx$ is the measure in $\Omega_d$, $\rd \bs$ the surface
measure on $\Gamma$, and $r$ the
local coordinate at $\bx \in \Gamma$ in the normal direction. Using \eqref{ass1},
we get
$
 \frac{1}{4} \leq \mu(\bx) \leq \frac{9}{4} \quad \text{for all} ~~\bx \in \Omega_d.
$
From this and relations \eqref{extension} and \eqref{cA}, we infer that $\alpha^e$  is strictly positive on a subset of $\Omega_d$ with positive measure:
\begin{equation*}
\widetilde{\mathcal{A}}:= \operatorname{meas}_{\bx}\{\bx\in\Omega_d\,:\, \alpha^e(\bx)\ge {\alpha}_0\}\ge
\frac89d\, \operatorname{meas}_{\bs}\{\bx\in\Gamma\,:\, \alpha(\bx)\ge {\alpha}_0\}>0,
\end{equation*}
with $\alpha_0>0$. Hence, similar to the surface case in~\eqref{Fried},  the Friedrich's type inequality
\begin{equation}\label{Fried1}
\|v\|_{L^2(\Omega_d)}^2 \le \widetilde{C}_F(\|\nabla v\|_{L^2(\Omega_d)}^2+\|\sqrt{\alpha} v\|_{L^2(\Omega_d)}^2)\quad
\forall~v\in H^1(\Omega_d)
\end{equation}
holds with a constant $\widetilde{C}_F$ dependent of $\alpha_0$ and $\widetilde{\mathcal{A}}$.

Inequalities \eqref{elp} and \eqref{Fried1} imply the ellipticity of the bilinear form:
$a(u,u)\ge c\|u\|^2$ for all $u\in H^1(\Omega_d)$, where the constant $c$ depends only on $\widetilde{C}_F$ from \eqref{Fried1}.
Therefore, part (i) of the theorem follows from the Lax-Milgram lemma.

To check part (ii) of the theorem, we note that $\Gamma\in C^3$ yields $\phi\in C^3$, see~\cite{DistFunc}, and $\dO_d\in C^3$. Therefore, the entries of the `diffusion' matrix $(\bI-\phi\bH)^{-2}$ are in $C^1$ and $\alpha\in L^\infty(\Gamma) \Rightarrow\alpha^e\in L^{\infty}(\Omega_d)$. This smoothness of the data is sufficient for the elliptic problem to be $H^2$-regular~\cite{ADN,Ladyzh} and the result follows.

\end{proof}

\begin{remark}\rm
 Theorem~\ref{Th1} shows one theoretical advantage of the new extended formulation  \eqref{ExtNew} over \eqref{1.1Ext} and \eqref{1.2}: If the data is smooth,  then  the Agmon-Douglis-Nirenberg regularity theory  immediately applies.  In particular, for $\Gamma\in C^3$, 
we have $u\in H^{2}(\Omega_d)$ and the trace theorem, see, e.g.,~\cite{Ladyzh}, yields $u|_\Gamma\in H^{1}(\Gamma)$.
This enables one to consider the trace of $u$ as the weak solution to \eqref{LBeq}.
\end{remark}


\section{Finite element method}\label{s_FEM}
Let $\Gamma\in C^2$ and fix a domain $\Omega_d$  such that the band width $d$ satisfies \eqref{ass1}. Assume $\T$ is a consistent
division (triangulation) of $\Omega_d$ into tetrahedra elements. We call a triangulation of $\Omega_d$   exact if $\bigcup_{T\in\T} \overline{T}=\overline{\Omega}_d$. Since $\dO_d$ coincides
with isolines of the distance function $\phi$, the boundary of $\Omega_d$ is curvilinear:  $\dO_d\in C^2$. Hence, exact triangulations of $\Omega_d$
may be constructed only in certain cases using isogeometric  elements~\cite{Isogeom} or mapped (blending) finite elements~\cite{blend}. In a general case,
we define the domain:
\[
\overline{\Omega}_h:=\bigcup_{T\in\T} \overline{T},
\]
which approximates $\Omega_d$.

Furthermore, in some applications the surface $\Gamma$ may not be known explicitly, but given only
approximately as, for example, the zero level set of a finite element distance function $\phi_h$.  In this case,
instead of the Hessian $\bH=\nabla^2 \phi$ one has to use a discrete Hessian $\bH_h\approx \bH$, which is  obtained from $\phi_h$ by any of discrete Hessian recovery methods, see e.g.~\cite{Hessian1,Hessian0}.
We assume that $\phi_h$ and $\bH_h$ satisfy condition~\eqref{ass1}.

Let $V_h\subset H^1(\Omega_h)$ be a space of finite element functions.
The finite element method reads: Find $u_h\in V_h$ satisfying
\begin{equation}\label{FEmeth}
\int_{\Omega_h}\big(\,(\bI-\phi_h\bH_h)^{-2}\nabla u_h\,\big)\cdot\nabla v_h + \alpha^e\,u_h v_h\,d\bx =\int_{\Omega_h} f^ev_h\,d\bx\quad\forall\,v_h\in V_h.
\end{equation}

We analyse the method \eqref{FEmeth} below in the special case of $\Omega_h=\Omega_d$, $\phi_h=\phi$, and $\bH_h=\bH$.
Numerical experiments in the next section test the method when non of these assumptions hold.

Since the diffusion tensor $(\bI-\phi\bH)^{-2}$ is uniform positive definite and bounded,  we immediately  obtain the following optimal convergence result, e.g.,~\cite{Braess}:
\begin{theorem}\label{Conv1} Let $\Omega_h=\Omega_d$, $\phi_h=\phi$, and $\bH_h=\bH$. Assume $u$ and $u_h$ solve
problems \eqref{ExtNew} and \eqref{FEmeth}, respectively. Then it holds
\[
\|u-u_h\|_{H^1(\Omega_d)} \le C \inf_{v_h\in V_h}\|u-v_h\|_{H^1(\Omega_d)},
\]
where the constant $C$ may depend only on $\alpha$, $\Gamma$ and $d$.
\end{theorem}

Theorem~\ref{Conv1} and the trace theorem yield the simple error estimate on the surface:
\begin{equation}\label{est_simple}
\|u-u_h\|_{L^2(\Gamma)} \le C\,\inf_{v_h\in V_h}\|u-v_h\|_{H^1(\Omega_d)}.
\end{equation}
However, the above estimate of the error in surface $L^2$ norm is not optimal and can be
improved. The improved estimate  is given by Theorem~\ref{ThMain}. To show it, we need several
preparatory results.

Denote \[h=\sup_{T\in\T}\mbox{diam}(T)\] and assume that $V_h$ is such that
\begin{equation}\label{approx}
\inf_{v_h\in V_h}\|v-v_h\|_{H^1(\Omega_d)}\le C_a\, h \|v\|_{H^2(\Omega_d)},\quad \forall\,v\in H^2(\Omega).
\end{equation}
The $L^2$-convergence estimate for the finite element method for the extended problem  is given in the next theorem.

\begin{theorem}\label{Conv2} Let $\Omega_h=\Omega_d$, $\phi_h=\phi$, $\bH_h=\bH$, and $\Gamma\in C^3$. Assume $u$ and $u_h$ solve
problems \eqref{ExtNew} and \eqref{FEmeth}, respectively. Then it holds
\[
\|u-u_h\|_{L^2(\Omega_d)} \le C\,h\,\inf_{v_h\in V_h}\|u-v_h\|_{H^1(\Omega_d)},
\]
where the constant $C$ may depend only on $\alpha$, $\Gamma$, $d$, and constant $C_a$ from \eqref{approx}.
\end{theorem}
\begin{proof} The assumption $\Gamma\in C^3$ ensures that the differential problem~\eqref{ExtNew}
is $H^2$-regular.  Since $\Omega_h=\Omega_d$, $\phi_h=\phi$, $\bH_h=\bH$, the discrete
problem \eqref{FEmeth} is the plain Galerkin method. Hence, the result follows from the standard
duality argument, see, e.g., \cite{Braess}.
\end{proof}

The result below is found for example in Theorem 1.4.3.1 of \cite{Gr85} (note that the unit simplex has uniformly Lipschitz boundary).

\begin{lemma}
\label{lem2-5}
Let $\widehat{T}$ be the unit simplex (triangle or tetrahedra) in $\mathbb{R}^N$, $N=2,3$.  Then there exists an extension operator $E:H^1(\widehat{T}) \rightarrow H^1(\mathbb{R}^N)$ such that
\begin{equation}
\label{2.8}
\|E v\|_{H^1(\mathbb{R}^N)} \le C\, \|v\|_{H^1(\widehat{T})}\quad \forall~v \in H^1(\widehat{T}).
\end{equation}
\end{lemma}

For a tetrahedron (triangle) $T$ denote by $\rho(T)$  the diameter of the inscribed ball. Let $\T_\Gamma$  be the set
of all tetrahedra intersected by $\Gamma$. Denote
\begin{equation}\label{beta}
\beta=\sup_{T\in\T_\Gamma}\mbox{diam}(T)/\rho(T).
\end{equation}
We assume that tetrahedra (triangles) in $\T_\Gamma$ are shape-regular, i.e., $\beta$ is not too big.
We need the following technical lemma.

\begin{lemma}\label{L_trace}
Let $T\in \T_\Gamma$. Denote $h=\mbox{diam}(T)$ and $\widetilde{K}=T\cap\Gamma$, then it holds
\begin{equation}\label{eq_trace}
\|v\|_{L^2(\widetilde{K})}\le C\,(h^{-\frac12}\|v\|_{L^2(T)}+h^{\frac12}\|\nabla v\|_{L^2(T)}),
\end{equation}
where the constant $C$ may depend only on $\Gamma$ and the constant $\beta$ from \eqref{beta}.
\end{lemma}
\begin{proof} The proof adopts the `flattening' argument from~\cite{DO12}, \S~3.4. The proof below is given for
the three-dimensional case: $\Gamma$ is a surface  in $\mathbb{R}^3$. All arguments remain valid with obvious modifications, if $\Gamma$ in a curve in $\mathbb{R}^2$.
We may assume that the curvilinear element $\widetilde{K}$ has non-zero 2D measure.
 Let $\widehat{T}$ be the reference unit tetrahedron in $\mathbb{R}^3$, let $\varphi:\,\widehat{T}\to T$ be an affine mapping with $\|\nabla \varphi\| \le c\, h $ and $\|(\nabla \varphi)^{-1}\| \le c\,h ^{-1}$. Here and in the rest of the proof, $c$ denotes a generic constant, which may depend only on $\beta$ and $\Gamma$, but does not depend on $T$.  We next recall that because $\Gamma$ is a $C^2$ surface, there exists a $C^2$ chart $\widetilde{\Phi}$ with uniformly bounded derivatives, and for which $\widetilde{\Phi}^{-1}$ has uniformly bounded derivatives, which maps an $O(1)$-neighborhood $N$ of $\widetilde{K}$ in $\mathbb{R}^3$ to $\mathbb{R}^3$ and which has the property that $\Gamma \cap N$ lies in a plane.  It is not difficult to extend $\widetilde{\Phi}$ to all of $\mathbb{R}^3$ so that the resulting extension has bounded derivatives, has a bounded inverse, and flattens an $O(1)$-neighborhood of $\widetilde{K}$.  We then define a corresponding flattening map for the reference space by $\Phi= \varphi^{-1} \circ \widetilde{\Phi} \circ \varphi$.  It is easy to check that then $\Phi$ and $\Phi^{-1}$ are also uniformly bounded in $C^2$, and $\Phi(\varphi^{-1}(\widetilde{K}))$ is flat.
Denote by $\mathbb{P}$  a plane in $\mathbb{R}^3$ containing the flattened surface element $\Phi ( \varphi^{-1} (\widetilde{K}))$.

We need the following trace inequality (\cite{Ad75}, Theorem 7.58):
\begin{equation}\label{trace}
\|v\|_{L^2(\mathbb{P})}\le \|v\|_{H^{\frac12}(\mathbb{P})}\le c\, \|v\|_{H^{1}(\mathbb{R}^3)}\quad
\forall~v\in H^{1}(\mathbb{R}^3).
\end{equation}

Define $\widehat{ v}$ on $\widehat{T}$ by $\widehat{ v}= v \circ \varphi$.  Given $T \in \T$, recalling the definition of the extension operator $E$ from Lemma \ref{lem2-5} and trace inequality~\eqref{trace}, we then compute
\begin{equation}
\label{3.6}
\begin{aligned}
h ^{-1} \| v \|_{L^2(\widetilde{K})} \le & c\,\| \widehat{  v} \|_{L^2(\varphi^{-1}(\widetilde{K}))}
\\  = &  c\,\| E \widehat { v} \|_{L^2(\varphi^{-1} (\widetilde{K}))}
\\  \le &  c\,\| E \widehat { v} \circ \Phi^{-1}  \|_{L^2(\Phi ( \varphi^{-1} (\widetilde{K})))}
\\ \le &  c\,\|E \widehat { v}\circ \Phi^{-1}  \|_{L^2(\mathbb{P})}.
 \\ \le & c\, \|E \widehat { v}\circ \Phi^{-1}  \|_{H^{1} (\mathbb{R}^3)}
 \\ \le & c\, \|E \widehat{ v}\|_{H^{1}(\mathbb{R}^3)}
 \\ \le & c\, \|\widehat{ v}\|_{H^{1}(\widehat{T})}.
 \end{aligned}
 \end{equation}
 Applying a scaling argument  yields
 \begin{equation}
 \label{3.8}
 \|\widehat{ v} \|_{H^{1}(\widehat{T})}
\le c(h_{T}^{-3/2} \|\widehat{ v} \|_{L^2(T)} + h^{-1/2} \|\nabla \widehat{ v} \|_{L^2(T)}).
\end{equation}
Estimates \eqref{3.6} and \eqref{3.8} prove the lemma.
\end{proof}

Summing up the estimate \eqref{eq_trace} over all elements from $\T_\Gamma$, we get for $v\in H^1(\Omega_d)$
 \begin{equation}
 \label{aux1}
\begin{aligned}
\|v\|_{L^2(\Gamma)}^2&\le C\,(h^{-1}_{\min}\sum_{T\in\T_\Gamma}\|v\|_{L^2(T)}^2+h_{\max}\sum_{T\in\T_\Gamma}\|\nabla v\|_{L^2(T)}^2)\\
&\le C\,(h^{-1}_{\min}\|v\|_{L^2(\Omega_d)}^2+h_{\max}\|\nabla v\|_{L^2(\Omega_d)}^2),
\end{aligned}
\end{equation}
with $h_{\min(\max)}=\min(\max)_{T\in\T_\Gamma}\mbox{diam}(T)$. For the next theorem, let us assume
$h_{\max}\le c\,h_{\min}$.

\begin{theorem}\label{ThMain} Let $\Omega_h=\Omega_d$, $\phi_h=\phi$, $\bH_h=\bH$, and $\Gamma\in C^3$. Assume $u$ and $u_h$ solve
problems \eqref{LBeq} and \eqref{FEmeth}, respectively. Then it holds
\[
\|u-u_h\|_{L^2(\Gamma)} \le C\,h^{\frac12}\inf_{v_h\in V_h}\|u-v_h\|_{H^1(\Omega_d)},
\]
where the constant $C$ may depend only on $\alpha$, $\Gamma$, $d$,  constants $C_a$ from \eqref{approx}, and $\beta$ from \eqref{beta}.
\end{theorem}
\begin{proof}
The result of the theorem follows from Theorems~\ref{Conv1},~\ref{Conv2} and \eqref{aux1}.
\end{proof}
\medskip

The error estimate from Theorem~\ref{ThMain} is an improvement of  \eqref{est_simple}, but still half-order suboptimal. The difficulty in proving the optimal estimate is that the analysis of the extended finite element problem in $H^1(\Omega_d)$ and $L^2(\Omega_d)$ norms gives little information of normal derivatives of the error, i.e. how accurate $u_h$ satisfies \eqref{const}.

Approaching optimal estimates of the error on $\Gamma$ is possible
applying the well-known results on interior maximum-norm estimates for finite element methods for elliptic problems~\cite{Linfty}.
This, however, requires some further restrictions on mesh and $V_h$.
To be precise, assume a quasi-uniform triangulation of $\Omega_h$ and let $\mbox{dist}(\partial\T_\Gamma,\partial\Omega_h)\ge c_1 d\ge c_0 h$, with $c_1$ independent of $h$ and
large enough constant $c_0$.
 Let $V_h$ be the space of $P_k$ finite element functions, $k\ge1$.
Now we apply Theorem~1.1 from \cite{Linfty} (in terms of  \cite{Linfty} we take the `basic' domain $\Omega_0=\T_\Gamma$ and the intermediate domain $\Omega_d=\Omega_h$). We  obtain
\begin{equation}\label{max_est1}
\begin{aligned}
\|u&-u_h\|_{L^{\infty}(\Gamma)}\le \|u-u_h\|_{L^{\infty}(\T_\Gamma)}\\
&\le C\left((\ln dh^{-1})^r \min_{v_h\in V_h}\|u-v_h\|_{L^{\infty}(\Omega_h)}
+ d^{-\frac{N}2}\|u-u_h\|_{L^2(\Omega_h)}\right)
\end{aligned}
\end{equation}
 for $u$ and $u_h$ solving \eqref{ExtNew} and \eqref{FEmeth}, respectively. Here $r=1$ for $k=1$ and $r=0$ for $k\ge2$.

Now we want to combine the result in \eqref{max_est1} with $L^2$ volume estimate from Theorem~\ref{Conv2}.
To do this, we have to assume $\Omega_h=\Omega_d$, which means curvilinear elements on the boundary of $\Omega_h$.
Although certain types of curvilinear elements are allowed by the analysis of   \cite{Linfty},
we avoid further assumptions on elements touching boundary, but simply separate from the boundary:
Consider $\Omega_h'=\{T\in\T~:~\overline{T}\cap\dO_d=\emptyset\}$. Assume $\Omega_h'$ consists only of shape-regular tetrahedra (triangles). 
The restriction of finite  element functions from $V_h$ on $\Omega_h'$ is denoted by $V_h(\Omega_h')$. We assume $V_h(\Omega_h')$ is the space of $P^k$ elements. If $d$ is fixed and  $h$ is sufficiently small,
then $\mbox{dist}(\partial\T_\Gamma,\partial\Omega_h')\ge c_1 d\ge c_0 h$, with $c_1$ independent of $h$ and large enough constant $c_0$. Hence, the result in \eqref{max_est1} holds with $\Omega_h$ replaced by $\Omega_h'$ and we can combine it with the estimate from Theorem~\ref{Conv2}.
 Thus, we  proved the following theorem.

 \begin{theorem}\label{ThMain2} Let $\Omega_h=\Omega_d$, $\phi_h=\phi$, $\bH_h=\bH$,  $\Gamma\in C^3$, $d$ is fixed such that  \eqref{ass1} holds, $h$ is sufficiently small, the triangulation of $\Omega_h'$ is quasi-uniform and consists of tetrahedra (triangles), and $V_h(\Omega_h')$ is  the space of $P^k$ elements, $k\ge1$. Assume $u$ and $u_h$ solve
problems \eqref{LBeq} and \eqref{FEmeth}, respectively. Then it holds
\[
\|u-u_h\|_{L^{\infty}(\Gamma)}\le C\left((\ln h^{-1})^r \min_{v_h\in V_h}\|u-v_h\|_{L^{\infty}(\Omega_h)}
+ h\,\inf_{v_h\in V_h}\|u-v_h\|_{H^1(\Omega_d)}\right),
\]
$r=1$ for $k=1$ and $r=0$ for $k\ge2$. The constant $C$ may depend only on $\alpha$, $\Gamma$, $d$,  constants $C_a$
from \eqref{approx}, and $\beta$ from \eqref{beta}.
\end{theorem}

As an example, assume $u$ is sufficiently smooth and $V_h$ is piecewise linear finite element space (mapped piecewise linear near $\dO_d$). Then Theorem~\ref{ThMain2} yields
optimal order convergence result (up to logarithmic term): $\|u-u_h\|_{L^{\infty}(\Gamma)}=O(h^2\ln h^{-1})$.

\section{Numerical examples}\label{s_numer}
In this section, we present results of several numerical experiments.
We start with the example of the  Laplace--Beltrami problem \eqref{LBeq}
on a unit circle in $\mathbb{R}^2$ with a known solution so that we are able to calculate
the error between the continuous and discrete solutions. We set $\alpha=1$ and consider
\[
u(r,\phi)=\cos(5\phi)
\]
in polar coordinates, similar to the Example~5.1 from~\cite{DDEH}.

For $d=0.05$, we build conforming quasi-uniform triangulation of $\Omega_d$ and  apply the regular refinement process.
The grid is always aligned with the boundary of $\Omega_d$ so that $\dO_h$ is an $O(h^2)$
approximation of $\dO_d$. The grid after one step of refinement in the upper right part of $\dO_h$  is shown
in Figure~\ref{fig0} (left).

\begin{figure}
\begin{center}
\includegraphics[scale=.17]{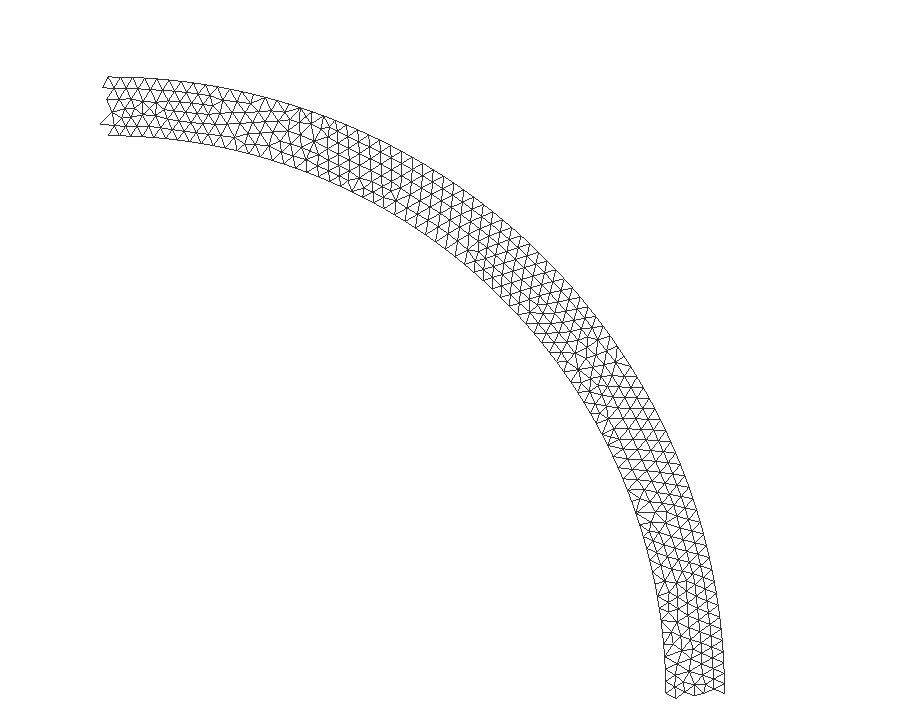}\qquad
\includegraphics[scale=.17]{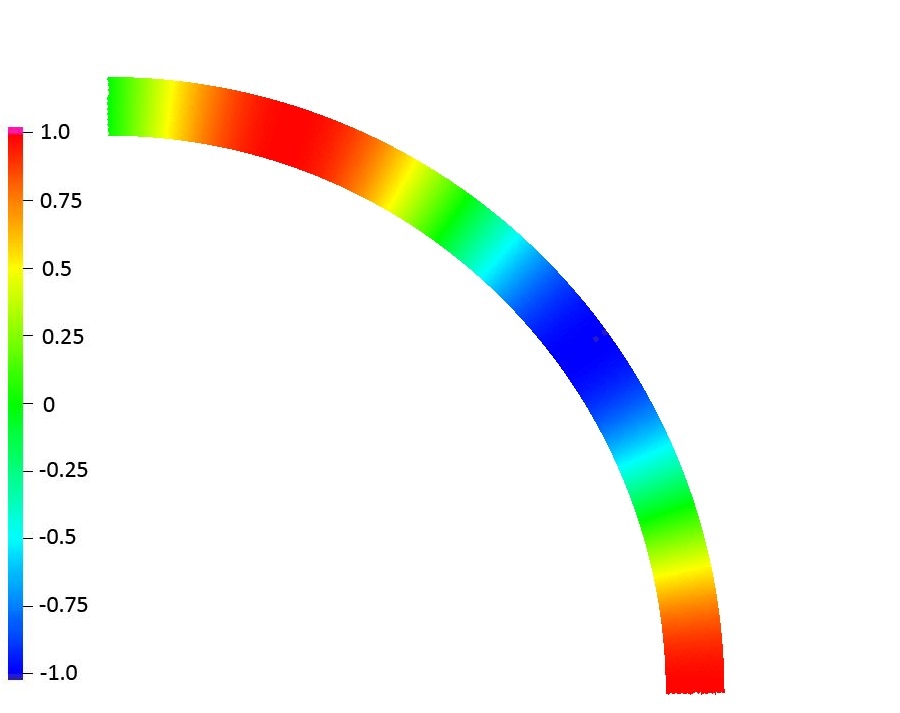}
\end{center}
\caption{The triangulation of $\Omega_d$ after one step of refinement for the 2D example
and the visualization of the computed  solution of the extended problem. \label{fig0}}
\end{figure}

The solution computed on grid level 3  is shown in Figure~\ref{fig0} (right). The visualized solution looks
constant in normal direction, as expected for solution of the extended problem. Next, we compute
the finite element error  restricted to the surface. For evaluating this error, we consider piecewise linear
approximations of $\Gamma$ and evaluate the errors along these piecewise linear surfaces.
To assess the estimates given by Theorems~\ref{ThMain} and~\ref{ThMain2},
we show in Table~\ref{tab0} the surface $L^2$ and $C$-norms of the errors. We clearly see the second order of
convergence in both norms, when the Hessian of the distance function is taken in \eqref{FEmeth} exactly.
We also experiment with approximate choice of $\phi_h$ and $\bH_h$ in \eqref{FEmeth}.
In this case, $\phi_h$ is a piecewise linear Lagrange interpolant to $\phi$, and $\bH_h$ is a piecewise linear continuous tensor-function recovered from $\phi_h$ by the variation method~\cite{Hessian2}. Compared to the exact
choice, the finite element errors are somewhat larger, although the convergence rates stay close to the second order.   The discrete problems were solved using the BCG method with the ILU2
preconditioner~\cite{Kaporin}  to a relative tolerance of $10^{-9}$. The iterations numbers are shown
in the right column of the table.

\begin{table}
\center
\caption{Norms of the errors  for the example of sphere with exact and approximate Hessian. $\Omega_d$ is fixed with $d=0.1$. \# Iter is the number of preconditioned BCG iterations.
} \label{tab0}
\footnotesize
\begin{tabular}{l|cr|rrrrr}\hline
& $h$  &\#d.o.f.  &$L^2$-norm & Order &$C$-norm & Order &\# Iter. \\ \hline\\[-2ex]
     & 0.0417	&610	 &0.318E-02	&       &0.345E-02&	&13\\
     & 0.0208	&2058    &0.662E-03&2.26	&0.148E-02&	1.22&28\\
$\bH$& 0.0104	&7351	 &0.179E-03&1.89    &0.308E-03&	2.26&60\\
     & 0.0052	&27954   &0.409E-04&2.13    &0.812E-04&	1.92&142\\
     & 0.0026	&109576  &0.983E-05&2.06	&0.195E-04&	2.06&325 \\
\hline \\[-2ex]
       & 0.0417	&610	&0.449E-02	&&0.451E-02&	&13\\
       & 0.0208	&2058&	0.147E-02	&1.61&0.182E-02&1.31	&28\\
$\bH_h$& 0.0104	&7351	&0.390E-03	&1.91&0.420E-03&2.16	&63\\
       & 0.0052	&27954&	0.124E-03	&1.65 &0.160E-03&1.39	&137\\
       & 0.0026	&109576&	0.325E-04&1.93	&0.425E-04&1.91	&297\\
\hline
\end{tabular}
\end{table}

As the next test problem,
 we consider the Laplace--Beltrami equation on the unit sphere:
\[
-\Delta_\Gamma u +u =f\quad\mbox{on}~\Gamma,
\]
with $\Gamma=\{ \bx\in\R^3\mid \|\bx\|_2 = 1\}$.
The source term $f$ is taken such that the solution is given by
\[
    u(\bx)= \frac{12}{\|\bx\|^3}\left(3x_1^2x_2 - x_2^3\right),\quad
    \bx=(x_1,x_2,x_3)\in\Omega.
\]
Note that $u$ and $f$ are constant along normals at $\Gamma$.

For different values of the domain width parameter $d$, we build conforming subdivisions of
$\Omega_d$ into regular-shaped tetrahedra using the  software package ANI3D \cite{ANI3D}.
The grid is aligned with the boundary of $\Omega_d$ so that $\dO_h$ is an $O(h^2)$
approximation of $\dO_d$.
The resulting discrete problems are again solved by the  BCG method with the ILU2
preconditioner  to a relative tolerance of $10^{-9}$.

\begin{table}
\center
\caption{Norms of the errors  for the example of sphere with exact and approximate Hessian. $\Omega_d$ is fixed with $d=0.1$. 
} \label{tab1}
\footnotesize
\begin{tabular}{l|r|rrrrr}\hline
&\#d.o.f. & $L^2$-norm & Order &$C$-norm & Order &\# Iter. \\ \hline\\[-2ex]
     &1026&	 0.6085E-01&     &	0.9033E-01&      &9	\\
$\bH$&8547&	 0.1503E-01&1.98 &	0.1523E-01& 2.52 &23  \\
     &63632&	 0.3990E-02&1.98 &	0.3971E-02& 2.01 &47  \\
\hline \\[-2ex]
       &1026&	0.8095E-01&	   &0.1032E+00&     &9	\\
$\bH_h$&8547&	0.2144E-01&1.88 &0.1909E-01&2.39&23  \\
       &63632& 0.5114E-02&2.14 &0.4529E-02&2.15 &46  \\
\hline
\end{tabular}
\end{table}

\begin{figure}
\begin{center}
\hfill\includegraphics[scale=.19]{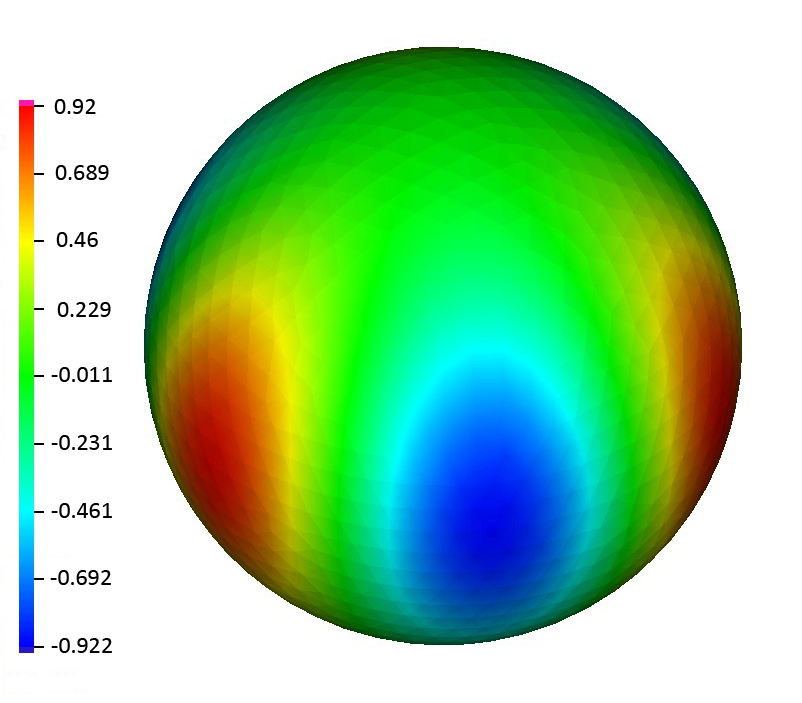}\hfill
\includegraphics[scale=.19]{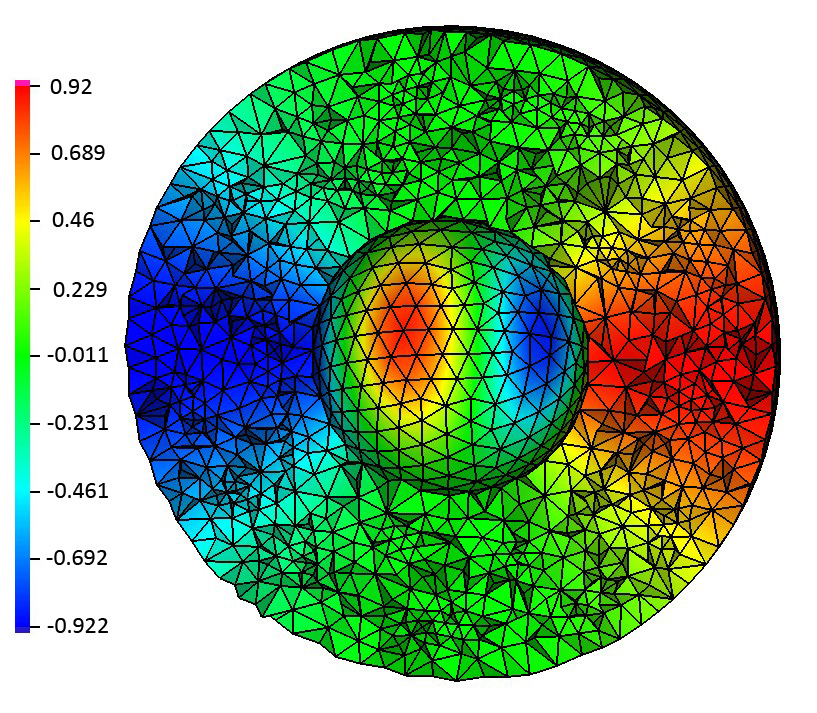}\hfill \\
\hfill\includegraphics[scale=.19]{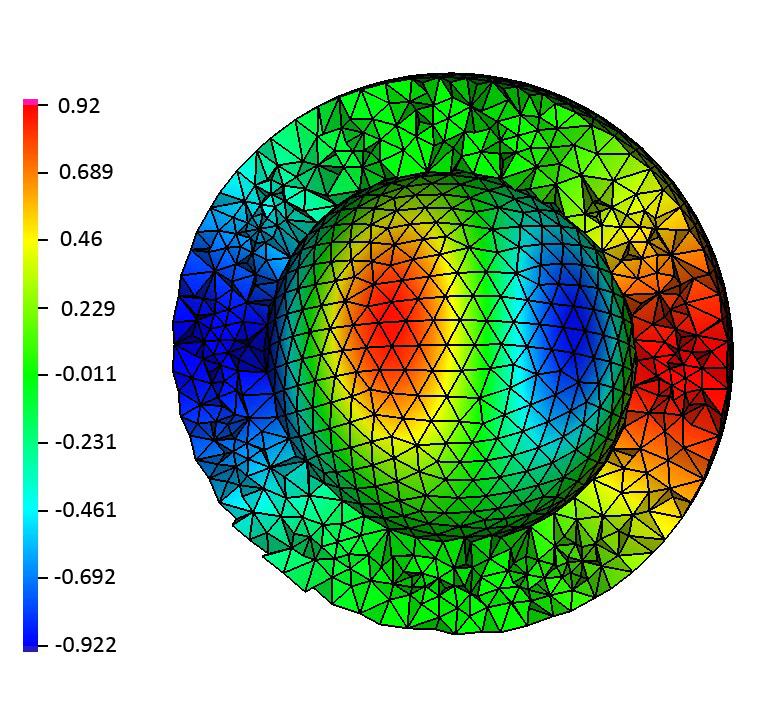}\hfill
\includegraphics[scale=.19]{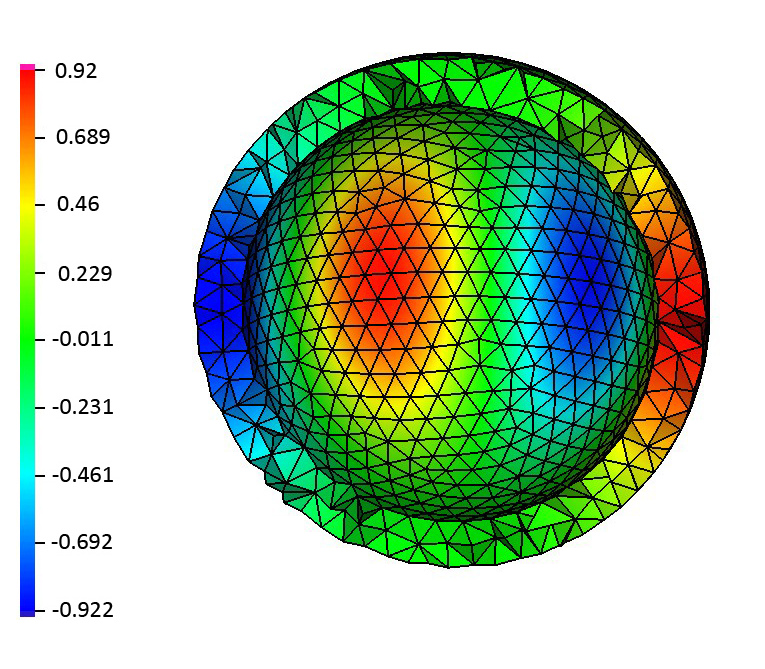}\hfill
\end{center}
\caption{The visualization of solution on the sphere and the cutaway of the volume grid
in $\Omega_d$ for $d=\{0.1,\, 0.2,\, 0.4\}$. \label{fig1}}
\end{figure}

\begin{table}
\begin{center}
\caption{Dependence of error norms on the domain width $d$ for the example with sphere. 
} \label{tab2}
\footnotesize
\begin{tabular}{rrrrr}\hline
$d$ &\#d.o.f. & $L^2$-norm & $C$-norm&  \# Iter. \\ \hline\\[-2ex]
0.4	&80442 &0.7700E-02	&0.6986E-02	&46 \\
0.2	&34305& 0.9389E-02	&0.9560E-02	&42 \\
0.1	&13560 &0.9579E-02	&0.1025E-01	&35 \\
\hline
\end{tabular}
\end{center}
\end{table}

In Tables~\ref{tab1}--\ref{tab2}, we show the norms of surface errors for the computed
finite element solutions and the number of preconditioned BCG iterations.
The surface errors were computed using the piecewise planar approximations of $\Gamma$
by $\Gamma_h$, where $\Gamma_h$ is the zero level set of the piecewise linear Lagrange interpolant
to the distance function of $\Gamma$. Table~\ref{tab1} shows the error norms for the case of a fixed
domain $\Omega_d$ and a sequence of discretizations. The formal convergence order $p$ was computed as
$$p=3\log\left(err_{1}/err_{2}\right)/\log\left((\#d.o.f.)_{1}/(\#d.o.f.)_{2}\right).$$
 We clearly observe the second order convergence
both when the exact distance function $\phi$ and the Hessian $\bH$ was used in \eqref{FEmeth}
and when $\phi$ and  $\bH$  were replaced by discrete $\phi_h$ and $\bH_h$. As in the two-dimension
case, $\phi_h$ is a piecewise linear Lagrange interpolant to $\phi$ and $\bH_h$ is a piecewise linear continuous tensor-function recovered from $\phi_h$ by the variation method.
In Table~\ref{tab2}, the results are shown for the case when the domain $\Omega_d$ shrinks towards the surface, while the mesh size $h$ was approximately the same for all three meshes.

The Figure~\ref{fig1} visualizes the solutions for various widths of the volume domains $\Omega_d$.
We see that the discrete solutions tend to be constant in the normal direction to the surface.

\begin{figure}

\begin{center}
\hfill\includegraphics[scale=.19]{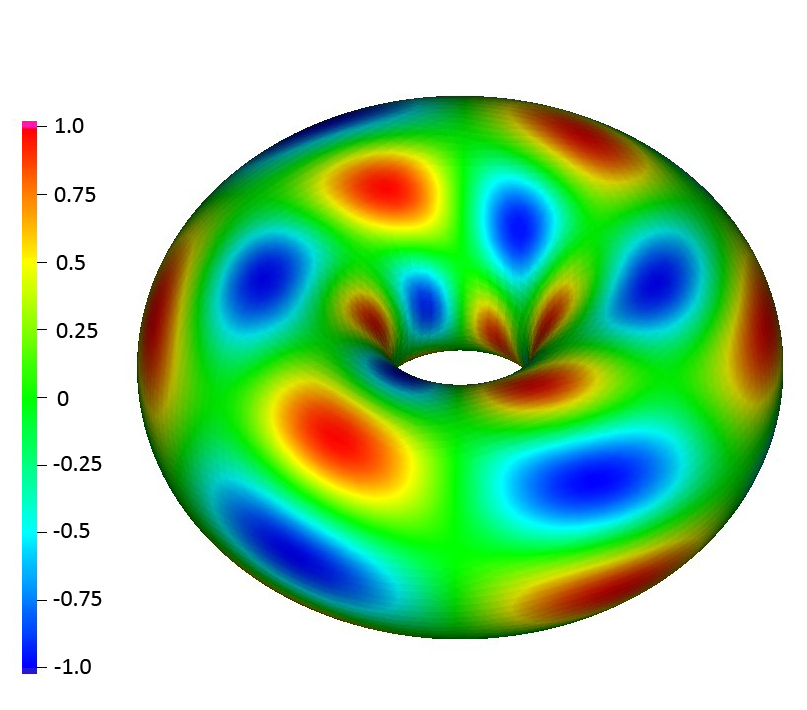}\hfill
\includegraphics[scale=.19]{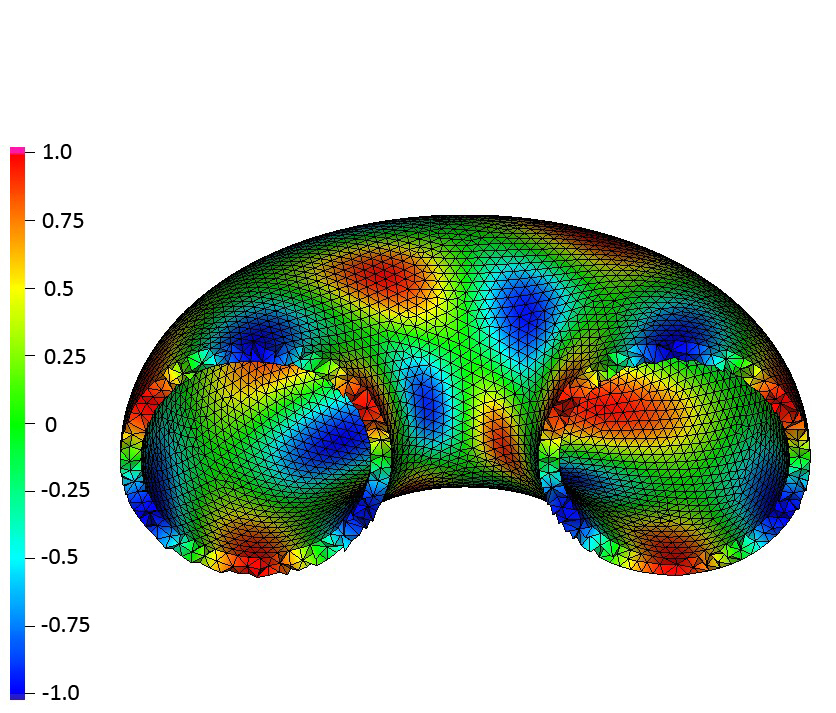}\hfill
\end{center}
\caption{The visualization of solution on the torus and the cutaway of the volume grid
in $\Omega_d$ for $d=0.1$. \label{fig2}}
\end{figure}

\begin{table}
\center
\caption{Norms of the errors  for the example of torus with exact and approximate Hessian. $\Omega_d$ is fixed with $d=0.1$. 
\label{tab3}}
\footnotesize
\begin{tabular}{l|r|rrrrr}\hline
&\#d.o.f. & $L^2$-norm & Order &$C$-norm & Order &\# Iter. \\ \hline\\[-2ex]
     &26257&   0.7826E-01&     &	0.1405E+00&      &18	\\
$\bH$&174021&  0.2843E-01&1.61 &	0.8400E-01& 0.82 &42  \\
     &1511742& 0.7780E-02&1.80 &	0.1077E-01& 2.85 &98  \\
\hline \\[-2ex]
       &26257&	    0.1144E+00&	    &0.1708E+00&     &20	\\
$\bH_h$&174021& 	0.7680E-01&0.63 &0.1569E+00&0.13&43  \\
       &1511742&    0.6888E-01&0.15 &0.9679E-01&0.67 &94  \\
\hline
\end{tabular}
\end{table}

\begin{figure}
\begin{center}
\hfill\includegraphics[scale=.18]{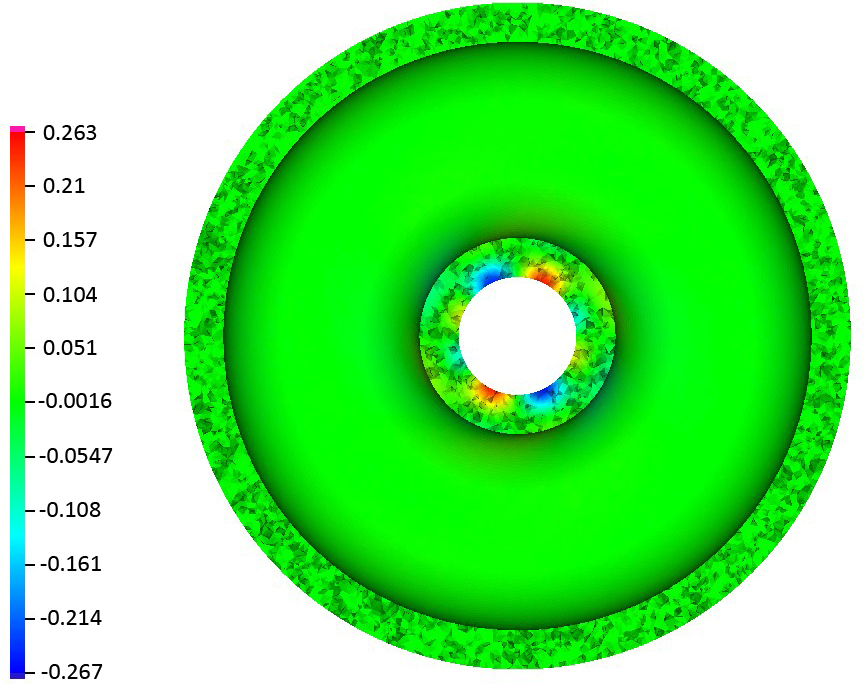}\hfill
\includegraphics[scale=.21]{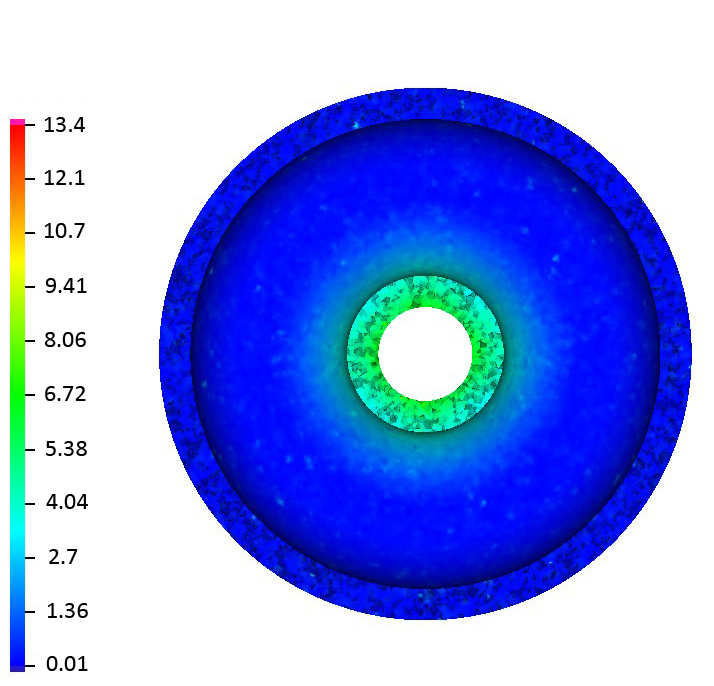}\hfill
\end{center}
\caption{Left: The finite element method error, with approximate Hessian.
Right: The error between discrete and approximate Hessian. \label{fig3}}
\end{figure}

We repeat the previous experiment,
but now with a  torus instead of the unit sphere.
Let $\Gamma = \{ \bx\in\Omega \mid r^2 = x_3^2 + (\sqrt{x_1^2 + x_2^2} - R)^2\}$.
We take $R= 1$ and $r= 0.6$. In the coordinate system $(\rho, \phi,
\theta)$, with
\[
    \bx = R\begin{pmatrix}\cos\phi \\ \sin\phi \\ 0 \end{pmatrix}
    + \rho\begin{pmatrix}\cos\phi\cos\theta \\ \sin\phi\cos\theta \\
    \sin\theta \end{pmatrix},
\]
the $\rho$-direction is normal to $\Gamma$, $\frac{\partial \bx}{\partial\rho}\perp\Gamma$ for $\bx\in\Gamma$.
The following solution $u$ and corresponding right-hand side
$f$ are constant in the normal direction:
\begin{equation}\label{polar2}
  \begin{split}
    u(\bx)&= \sin(3\phi)\cos( 3\theta + \phi),\\
    f(\bx)&= r^{-2} (9\sin( 3\phi)\cos( 3\theta + \phi)) \\
          & \quad - (R + r\cos( \theta)^{-2}(-10\sin( 3\phi)\cos(3\theta + \phi) - 6\cos( 3\phi)\sin( 3\theta + \phi)) \\
          & \quad -(r(R + r\cos( \theta))^{-1}(3\sin( \theta)\sin(
          3\phi)\sin( 3\theta + \phi)).
  \end{split}
\end{equation}

The surface norms of approximation errors for the example of torus are given in Table~\ref{tab3}. The solution is visualized in Figure~\ref{fig2}.
 Again, when  the exact Hessian is used, the convergence order is close to the second one.
However, when the exact Hessian is replaced by the recovered Hessian, the convergence significantly
deteriorates. This is opposite to the example with the sphere. A closer inspection reveals that the error is concentrated in the proximity of the inner ring of the torus, where the Gauss curvature is negative (see Figure~\ref{fig3}, left). Next, we look on the error of the discrete Hessian recovery:
$|\bH-\bH_h|:=\left(\sum_{i,j=1}^2 (\bH-\bH_h)^3_{i,j}\right)^{\frac12}$. The Figure~\ref{fig3}, right,
shows that the error  $|\bH-\bH_h|$ is large at the same region, near the inner ring of the torus. At this part of $\Omega_d$
 the Hessian is indefinite, it has non-zero values of different signs.

\section{Conclusions} \label{s_concl}

We studied a formulation and a finite element method for elliptic partial differential equation
posed on hypersurfaces in $\mathbb{R}^N$, $N=2,3$. The formulation uses an extension of the equation off the surface to a volume domain containing the surface. Unlike the original method from  \cite{BCOS01}, the extension introduced in the paper results in uniformly elliptic problems in the volume domain.  This enables a straightforward application of standard discretization techniques and put the problem into a well-established framework for analysis of elliptic PDEs, including numerical analysis.
For the standard Galerkin finite element
method we proved  new convergence estimates in the surface $L^2$ and $L^\infty$ norm.
Optimal convergence of the P1 finite element method was demonstrated numerically. The method, however, requires a reasonably good approximation of the Hessian for the signed distance function to the surface.

\section*{Acknowledgments}
This work has been supported in part by the Russian Foundation for Basic Research through grants 12-01-91330, 12-01-00283, 11-01-00971 and by the Russian Academy of Science  program ``Contemporary problems of theoretical mathematics'' through the project No. 01.2.00104588.


\end{document}